\documentclass[11pt, pb-diagram]{amsart}\usepackage{pb-diagram} 

\usepackage[utf8]{inputenc} 

\newtheorem{theorem}{Theorem}[section]
\newtheorem{corollary}[theorem]{Corollary}

\newtheorem{lemma}[theorem]{Lemma}
\newtheorem{example}{Example}

\theoremstyle{definition}
\newtheorem{definition}[theorem]{Definition}
\newtheorem{remark}[theorem]{Remark}

\usepackage{graphicx} 

\title { Poincar\'e maps and suspension flows: a categorical remark }
\author{Tomoharu Suda}
\address{Faculty of Mathematics, Keio University}
\email{tomoharu.suda@keio.jp}

\begin{document}

\maketitle 

\begin{abstract}
Poincar\'e maps and suspension flows are examples of fundamental constructions in the study of dynamical systems. 
This study aimed to show that these constructions define an adjoint pair of functors if categories of dynamical systems are suitably set. First, we consider the construction of Poincar\'e maps in the category of flows on topological manifolds, which are not necessarily smooth. We show that well-known results can be generalized and the construction of Poincar\'e maps is functorial, if a category of flows with global Poincar\'e sections is adequately defined. Next, we consider the construction of suspension flows and its functoriality. Finally, we consider the adjointness of the constructions of Poincar\'e maps and suspension flows. By considering the naturality, we can conclude that the concepts of topological equivalence or topological conjugacy of flows are not sufficient to describe the correspondence between map dynamical systems and flows with global Poincar\'e sections. We define another category of flows with global Poincar\'e sections and show that the suspension functor and the Poincar\'e map functor form an adjoint equivalence if these categories are considered. Hence, a categorical correspondence between map dynamical systems and flows with global Poincar\'e sections is obtained. This will enable us to better understand the connection between map dynamical systems and flows. 
\end{abstract}

\section{Introduction }

Poincar\'e map and suspension flow constructions are fundamental tools employed in the study of dynamical systems. They are used to reduce a problem concerning continuous-time systems to one of discrete-time systems or vice versa, thereby connecting the two major types of dynamical systems \cite{hasselblatt2002handbook, katok1997introduction, robinson1998dynamical}. 

Results on their relationship are scattered across the literature, and systematic treatments are scarce. However, by collecting these  results, we can easily observe that a categorical relationship may exist between them. For example, the following properties are known:
\begin{itemize}
	\item If two diffeomorphisms are topologically conjugate, then their suspensions are topologically conjugate (Proposition 5.38 in \cite{irwin2001smooth}).
	\item A flow with a Poincar\'e section is locally topologically equivalent to the suspension of its Poincar\'e map (Theorem 5.40 in \cite{irwin2001smooth}).
	\item Every diffeomorphism on a compact manifold is topologically conjugate with the Poincar\'e map of its suspension (Proposition 3.7 in \cite{palis2012geometric}).
\end{itemize}
 In the case of flows with global sections, stronger properties hold because Poincar\'e maps can be defined globally:
\begin{itemize}
	\item Topological equivalence of two flows can be determined in terms of Poincar\'e maps (Theorem 1 in \cite{basener2002global}, Proposition 1.11 in \cite{phdthesis}).
	\item A flow with a global section is topologically equivalent to the suspension of its Poincare map (Theorem 3.1 in \cite{yang2000remark}).
\end{itemize}
In loose terms, these results can be summarized as follows: isomorphisms are preserved under the constructions of Poincar\'e maps and suspension flows, and a Poincar\'e map of a suspension or a suspension of a Poincar\'e map can be identified with the original map or flow. These statements  suggest the existence of categorical equivalence between a category of map dynamical systems and one of flows .

Some categorical aspects of these constructions have been considered in the case of isomorphisms with topological conjugacy \cite{cestau2017prolongations}. However, their relation remains unclear because it depends on the choice of categories.  For example, some of the results mentioned above are not true if one uses topological conjugacy instead of topological equivalence to define isomorphisms. 

This study aimed to perform a categorical treatment of the constructions of Poincar\'e maps and suspension flows in order to describe the exact relationship between them. This will enable us to unify the known results listed above and also ``prove" the folklore correspondence of various notions between discrete-time and continuous-time systems, such as that of topological conjugacy and topological equivalence.

The rest of this paper is organized as follows. In Section \ref{sec_cat}, we define several categories of dynamical systems. In Section \ref{sec_poincare}, we first introduce the notion of topological transversality for topological manifolds and continuous flows. We show that Poincar\'e maps can be defined analogously to the smooth case. Then, we define categories of flows with global Poincar\'e sections to show that the construction of Poincar\'e maps is functorial. In Section \ref{sec_ps}, we study the categorical relationship between Poincar\'e maps and suspension flows. We show that these two form a pair of adjoint equivalence if the categories are selected properly. Finally, in Section \ref{sec_conclude}, we present some concluding remarks.

\section{Categories  of dynamical systems}\label{sec_cat}
In this section, we define various categories of dynamical systems to set up for the discussion later.

In what follows, topological manifolds are assumed to be second countable and Hausdorff. For the definitions of the concepts and basic results of category theory, we refer to \cite{mac1998categories, bradley2020topology, riehl2017category}.
\begin{definition}
A \emph{map dynamical system} is a pair $(f,X)$ of a topological manifold (without boundary) $X$ and a homeomorphism $f:X \to X.$
A \emph{morphism} $h:(f,X) \to (g,Y)$ between map dynamical systems is a continuous map $h: X \to Y$ such that  $h \circ f = g \circ h.$
\end{definition}
\begin{definition}
A \emph{flow} is a pair $(\Phi,X)$ of a topological manifold (without boundary) $X$ and a continuous map $\Phi:X\times \mathbb{R} \to X$ such that
\begin{enumerate}
	\item For each $x \in X,$ $\Phi(x,0) = x.$
	\item For each $x \in X$ and $s, t \in \mathbb{R},$ $\Phi(\Phi(x,t),s) = \Phi(x, t+s).$ 
\end{enumerate}
A \emph{morphism} $h:(\Phi,X) \to (\Psi,Y)$ between flows is a continuous map $h: X \to Y$ such that $h \left( \Phi(x,t) \right) =  \Phi(h(x),t)$ for all $x \in X$ and $t \in \mathbb{R}.$

A \emph{weak morphism} $(h, \tau):(\Phi,X) \to (\Psi,Y)$ between flows is a pair of a continuous map $h: X \to Y$ and a map $\tau: X \times \mathbb{R} \to \mathbb{R}$ such that 
\begin{enumerate}
	\item $h \left( \Phi(x,t) \right) =  \Psi(h(x),\tau(x,t))$ for all $x \in X$ and $t \in \mathbb{R}.$
	\item For all $x \in X,$ $\tau(x, -): \mathbb{R} \to \mathbb{R}$ is an increasing homeomorphism with $\tau(x,0) = 0.$
	\end{enumerate}
\end{definition}
\begin{lemma}
Each of the following  forms a category if the composition of morphisms is defined by the composition of maps.
\begin{enumerate}
	\item Map dynamical systems and their morphisms. 
	\item Flows and their morphisms.
	\item Flows and weak morphisms.
\end{enumerate}

\end{lemma}
\begin{proof}
The proof  is obvious for (1) and (2). For (3), we need to verify that the "time-part" composition of the morphism satisfies the conditions of weak morphism. Let $(h_1, \tau_1):(\Phi_1,X_1) \to (\Phi_2,X_2)$ and $(h_2, \tau_2):(\Phi_2,X_2) \to (\Phi_3,X_3)$ be morphisms and define $\tau_2 \circ \tau_1(x,t): = \tau_2(h_1(x), \tau_1(x,t)).$ Then, for all $x \in X_1,$ $\tau_2 \circ \tau_1(x, 0) =  \tau_2(h_1(x), 0) = 0$ and $\tau_2 \circ \tau_1(x, -) $ is a composition of homeomorphisms.
\end{proof}
We call the above a \emph{category of map dynamical systems} $\bf{Map}$, a \emph{category of flows} $\bf{Flow}$, and a \emph{category of flows with weak morphisms} $\bf{WFlow},$, respectively. 
We note that $\bf{Flow}$ can be regarded as a subcategory of $\bf{WFlow},$ as there is an obvious inclusion functor defined by $(\Phi, X) \mapsto (\Phi, X) $ and $\left((\Phi,X) \xrightarrow{h}  (\Psi,Y) \right)\mapsto \left((\Phi,X) \xrightarrow{(h, \rm{id})}  (\Psi,Y) \right).$

Isomorphisms in $\bf{Map}$ and $\bf{Flow}$ are called \emph{topological conjugacies} and isomorphic objects are called  \emph{topologically conjugate}. In $\bf{WFlow}$, isomorphism is called \emph{topological equivalence} and isomorphic objects are called \emph{topologically equivalent}. These definitions coincide with the usual ones.

\begin{remark}
The categories in \cite{cestau2017prolongations} correspond to $\bf{Map}$ or $\bf{Flow}$ in this paper.
\end{remark}
\begin{remark}
Each of the categories defined above has a weakly initial element similar to the ``universal dynamical system" of \cite{riehl2017category}. For example, the system $(\sigma, \mathbb{Z})$ defined by 
\[
	\sigma(n) = n+1
\] 
for all $n \in \mathbb{Z}$ is weakly initial in the category $\bf{Map}.$ Further, the set $\mathbf{Map}\left( (\sigma, \mathbb{Z}), (f, X)\right)$ is isomorphic to the set of all orbits of $(f, X).$ In particular, a morphism $h: (\sigma, \mathbb{Z}) \to (f, X)$ corresponds to an orbit with period $m \in \mathbb{N}$ if and only if $h$ admits the following factorization:
\[
\begin{diagram}
	\node{(\sigma, \mathbb{Z})}
	\arrow{s,t}{}
	\arrow{se,t}{h}\\
	\node{(\sigma, \mathbb{Z}/ m \mathbb{Z})}
	\arrow{e,t}{\bar h}
	\node{(f, X)}
\end{diagram}
\]
Similar constructions can be carried out for $\bf{Flow}$ or $\bf{WFlow}.$
\end{remark}
%
\section{Topological transversality and global Poincar\'e section}\label{sec_poincare}
In this section, we define the concept of topological transversality for continuous flows on topological manifolds. Based on this definition, we show that Poincar\'e maps can be defined in a manner similar to the classical smooth case. Additionally, we introduce categories of flows with global Poincar\'e sections to consider the functoriality of the construction of Poincar\'e maps.

We adopt the definition of topological transversality given in \cite{phdthesis, basener2004every} with a certain modification.
\begin{definition}\label{Top_trans}
Let  $(\Phi, X)$ be a flow, where $X$ is an $n$-dimensional topological manifold. A submanifold $S \subset X$ without a boundary is \emph{topologically transversal} to $\Phi$ if
\begin{enumerate}
	\item $S$ is codimension one and locally flat.
	\item For each $x \in S,$ there exists a neighborhood $U$ of $x$ in $X$ and a homeomorphism $\phi: U \to B \subset \mathbb{R}^n,$ where $B$ is the unit ball such that $\Phi\left(U\cap S \right) = B \cap \mathbb{R}^{n-1}\times \{0\}.$ Further, there exist $\delta_{+}(x) >0$ and $\delta_{-}(x) <0$ such that $\Phi(x, [\delta_{-}(x),0))$ and $\Phi(x, (0,\delta_{+}(x)])$ are contained in different connected components of $U \backslash S$ and $\Phi(x, [\delta_{-}(x),\delta_{+}(x)])) \cap S = \{x\}.$
	Here, $\delta_{+}$ and $\delta_{-}$ can be taken locally uniformly, that is, there exist a neighborhood $V \subset U$ of $x$ and $\delta >0$ such that $\delta_+(y) > \delta$ and $\delta_-(y) < - \delta$ for all $y \in V\cap S.$
	\item For each set of the form $\Phi(y, [a,b]),$ where $y \in X$ and $a, b \in \mathbb{R},$ $\Phi(y, [a,b])\cap S$ is compact in $S$.
\end{enumerate} 
\end{definition}

\begin{lemma}\label{lem_intersection}
Let  $(\Phi, X)$ be a flow and $S \subset X$ be topologically transversal to $\Phi$. Then, for each $x\in S$ and $\epsilon>0,$ there exists an open neighborhood $V$ of $x$ in $X$ such that 
\[
	\Phi(y, [ - \epsilon, \epsilon]) \cap S \neq \emptyset
\]
 for all $y \in V.$
\end{lemma}
\begin{proof}
Let $U$ be a neighborhood of $x$ satisfying the condition of (2) in Definition \ref{Top_trans}. 
By the continuity of $\Phi$, there exist a neighborhood $V_0$ of $x$ and $\delta > 0$ such that $\Phi(V_0, [-\delta, \delta]) \subset U$ and $\delta < \min\left(\epsilon, \delta_+(x), -\delta_-(x)\right).$  Then, we have $\Phi(x, \delta), \Phi(x, -\delta)  \in U\backslash S.$ We take neighborhoods $V_+$ and $ V_-$ of $\Phi(x, \delta)$ and $ \Phi(x, -\delta),$ respectively, such that $V_+ \subset U\backslash S$ and $V_-  \subset U\backslash S.$ Let $V := \Phi(V_+, -\delta) \cap  \Phi(V_-, \delta) \cap V_0.$ By considering the $n$-th coordinate of the homeomorphism $\phi$, we obtain that $\Phi(y, [-\epsilon, \epsilon]) \cap S \neq\emptyset$ for all $y \in V.$
\end{proof}

The next lemma excludes the possibility of sequences that  return to the section very frequently.
\begin{lemma}\label{no_freq_seq}
Let  $(\Phi, X)$ be a flow and $S \subset X$ be topologically transversal to $\Phi$. If $x \in S$, there exist no sequences $x_n \in S$ and $t_n >0$ such that $x_n \to x$ and $t_n \to 0$ as $n\to \infty$ and $\Phi(x_n, t_n) \in S.$
\end{lemma}
\begin{proof}
Let $U$ be a neighborhood of $x$ satisfying condition (2) in Definition \ref{Top_trans}. Let $V \subset U$ be a neighborhood of $x$ such that there exists $\delta >0$ with $\delta_+(y) > \delta$ for all $y \in V \cap S.$ If $x_n \in S$ and $t_n >0$ are sequences such that $x_n \to x$ and $t_n \to 0$ as $n\to \infty$ and $\Phi(x_n, t_n) \in S,$ then $x_n \in V \cap S$ and consequently $t_n > \delta$ for a sufficiently large $n$. This is a contradiction.
\end{proof}
\begin{definition}
Let  $(\Phi, X)$ be a flow. A submanifold $S \subset X$ is a \emph{global Poincar\'e section} if
\begin{enumerate}
	\item $S$ is topologically transversal to $\Phi.$ 
	\item For each $x \in X,$ there exists $t_+ >0$ and $t_- < 0$ such that $\Phi(x, t_{+}) \in S$ and $\Phi(x, t_{-}) \in S$.
\end{enumerate}
\end{definition}
\begin{remark}
If the phase space is compact, condition (2) can be weakened to the condition that each $x \in X$ has  $t \in \mathbb{R}$ such that $\Phi(x, t) \in S.$ Indeed, let $x \in S$ and consider the $\omega$-limit set of $x.$ Then, $\omega(x) \cap S$ is nonempty by the invariance of the limit set. By Lemma \ref{lem_intersection}, we observe that there exists $t_+ > 0$ such that $\Phi(x, t_{+}) \in S.$ The existence of $t_-$ is proved similarly.
\end{remark}
\begin{remark}
By definition, a flow with a global Poincar\'e section has no equilibrium points. By using the argument in \cite{basener2002global}, we can show that a smooth flow without equilibrium points has a global Poincar\'e section if the phase space is compact.
\end{remark}
According to these definitions, we have the following generalization of well-known results.
\begin{theorem}\label{P_thm}
Let  $(\Phi, X)$ be a flow  and $S \subset X$ be topologically transversal to $\Phi$. If $x_0 \in S$ and there exists $t_+ > 0$ such that  $\Phi(x_0,t_+) \in S,$ there exist a neighborhood $U$ of $x_0$ in $X$ and continuous maps $P\Phi: U\cap S \to S$ and $T_\Phi:U\cap S \to (0,\infty)$ such that
\[ P\Phi(x) = \Phi(x, T_\Phi(x))\]
for each $x \in U\cap S$ and $\Phi(x, t) \not \in S$ for $0<t<T_\Phi(x).$
Further, if $S$ is a global Poincar\'e section, $P \Phi$ is defined on the entire $S$, and it is a homeomorphism.
\end{theorem}
\begin{proof}
First, we show the existence of $T_\Phi(x)$ and $P \Phi(x)$ for each $x \in U\cap S,$ where $U$ is a neighborhood of $x_0$ in $X.$ Let $0<r< t_+$ and take a neighborhood $V$ of $\Phi(x_0,t_+)$ by applying Lemma  \ref{lem_intersection} with $\epsilon = r$ and $x = \Phi(x_0,t_+).$ Let $U := \Phi(V, -t_+).$ Then, we have
\[
	\Phi(y, [t_+ -r, t_+ +r]) \cap S \neq \emptyset
\]
for all $y \in U.$ We define
\[
	\mathcal{T}(x) := \{t >0 \mid \Phi(x,t) \in S\}
\]
for each $x \in U\cap S.$
By the choice of $U$, $\mathcal{T}(x)$ is nonempty. If we set $T_\Phi(x) := \inf \mathcal{T}(x),$ we have $T_\Phi(x) \geq \delta_+(x)>0.$ Let $t_n \in \mathcal{T}(x)$ be a sequence with $t_n \to T_\Phi(x) $ as $n\to \infty.$ For a sufficiently large $a >0,$ we have $\Phi(x, t_n) \in S \cap \Phi(x, [0,a])$ for all $n.$ Therefore, $\Phi(x, T_\Phi(x)) \in S$ from the continuity of $\Phi$ and the compactness of $S \cap \Phi(x, [0,a]).$ These results indicate that $T_\Phi(x)$ has the desired properties. We set $P\Phi(x) := \Phi(x, T_\Phi(x)).$

Let us now show that $T_\Phi : U\cap S \to (0,\infty)$ is continuous. Let $x \in U\cap S$ and $\epsilon $ be an arbitrary positive number less than $T_\Phi(x)$. 

By Lemma \ref{lem_intersection}, there exists an open neighborhood $V_1 \subset U$ of $P\Phi (x) = \Phi(x, T_\Phi(x))$ such that  $\Phi(y, [ - \epsilon,  \epsilon]) \cap S \neq \emptyset$ for all $y \in V_1.$ By the continuity of $\Phi,$ there exists an open neighborhood $U_1$ of $x$ such that  $\Phi(U_1, T_\Phi(x)) \subset V_1$. Therefore, we have
\[
	\Phi(y, [T_\Phi(x) - \epsilon, T_\Phi(x) + \epsilon]) \cap S \neq \emptyset
\]
for all $y \in U_1.$ 

We show that there exists an open neighborhood $U_2 \subset U$ of $x$ such that 
\[\Phi(y, (0, T_\Phi(x) - \epsilon)) \cap S = \emptyset\]
 for all $y \in U_2 \cap S.$ If this is not the case, we may take sequences $x_n \in S \cap U$ and $s_n \in (0, T_\Phi(x) - \epsilon)$ so that $\Phi(x_n, s_n) \in S$ and $x_n \to x$ as $n \to \infty.$ As $s_n \in [0, T_\Phi(x) - \epsilon],$ we may take a convergent subsequence $s_{n_i} \to s \in [0, T_\Phi(x) - \epsilon]$ as $i \to \infty.$ Using the continuity of $\Phi$, we observe that $s = 0.$ Thus, we obtain sequences $y_n \in S\cap U$ and $t_n \in (0, T_\Phi(x) - \epsilon)$ so that $y_n \to x$ and $t_n \to 0$ as $n \to \infty.$ However, this is impossible by Lemma \ref{no_freq_seq}.

Therefore, there exists an open neighborhood $U_0:= U_1 \cap U_2 \subset U$ of $x$ such that $\Phi(y, (0, T_\Phi(x) - \epsilon)) \cap S = \emptyset$  and $\Phi(y, [T_\Phi(x) - \epsilon, T_\Phi(x) + \epsilon]) \cap S \neq \emptyset$ for all $y \in U_0 \cap S.$ Together, these imply 
\[
	T_\Phi(x) - \epsilon \leq T_\Phi(y) \leq T_\Phi(x) + \epsilon
\]
for all $y \in U_0 \cap S.$ Therefore, $T_\Phi(x)$ is continuous and consequently $P \Phi(x)$ is also continuous.

If $S$ is a global Poincar\'e section, it is clear that $T_\Phi$ and $P \Phi $ are defined on the entire $S$.
By the definition of a global Poincar\'e section, the same constructions can be carried out for $\Psi(x,t) := \Phi(x, -t).$ Then, we have 
\[
	\begin{aligned}
		T_\Psi  &= T_\Phi \circ P\Psi\\
		T_\Phi &= T_\Psi \circ P \Phi.
	\end{aligned}
\]
These are established as follows. For $x \in S,$ we have $\Phi(P\Psi(x), T_\Psi(x)) = x \in S$ by definition. Therefore, $T_\Psi(x) \geq T_\Phi(P\Psi(x)).$ On the other hand, we have
$\Phi(x, t) \not \in S$ for $-T_\Psi(x)<t<0$ by definition of $T_\Psi(x)$. This implies $\Phi(P\Psi(x),t) \not \in S$ for $0<t<T_\Psi(x).$ Therefore, $T_\Psi(x) \leq T_\Phi(P\Psi(x)).$ The other relation is obtained by symmetry. Now, we have $(P\Phi)^{-1} = P\Psi$ because
\[
\begin{aligned}
	P\Phi(P\Psi(x)) &= \Phi(P\Psi(x), T_\Phi(P\Psi(x))) = \Phi(P\Psi(x), T_\Psi(x)) =x\\
	P\Psi(P\Phi(x)) &= \Psi(P\Phi(x), T_\Psi(P\Phi(x))) = \Psi(P\Phi(x), T_\Phi(x)) =x
\end{aligned}
\]
for all $x \in S.$ 
	\end{proof}
\begin{corollary}\label{div_sum}
Let $(\Phi, X)$ be a flow with a global Poincar\'e section $S.$ If $x \in S,$ we have
\[
	\begin{aligned}
		\sum_{n=0}^\infty T_\Phi \circ (P\Phi)^n(x) & = \infty \\
		\sum_{n=1}^{\infty} T_\Phi \circ (P\Phi)^{-n}(x) & = \infty.
	\end{aligned}
\]
\end{corollary}
\begin{proof}
As we have $T_\Psi  = T_\Phi \circ P\Psi =T_\Phi \circ (P\Phi)^{-1}$, it is sufficient to prove  the first formula. Suppose the series is convergent and let the sum $T_\infty$ and $x_\infty:= \Phi(x, T_\infty).$ Because we have
\[
	\Phi\left(x, \sum_{n=0}^{N-1} T_\Phi \circ (P\Phi)^n(x)\right) = (P\Phi)^{N} (x)
\]
for each $N \geq 1,$ $x_n := (P \Phi)^n(x)$ converges to $x_\infty.$ For all $n,$ $x_n$ is contained in $\Phi(x,[0, T_\infty])\cap S,$ which is compact by the definition of topological transversality. Therefore, $x_\infty \in S.$ If we set $t_n := T_\Phi \circ (P\Phi)^n (x),$ the convergence of the sum implies $t_n \to 0$ and $\Phi(x_n, t_n) \in S$ for all $n$ by definition. Thus, we have a pair of sequences $(x_n, t_n),$ which does not exist by Lemma \ref{no_freq_seq}. This is a contradiction.
\end{proof}
A flow may admit many different global Poincar\'e sections, and consequently, a pair of a flow and a section may not necessarily be preserved under a weak morphism. If the sections are preserved by a weak morphism as sets, we have the following correspondence of the first return times between two flows.
\begin{lemma}\label{lem_period}
Let $(\Phi, X)$ and $(\Psi, Y)$ be flows with global Poincar\'e sections $S$ and $S',$ respectively, and $(h, \tau): (\Phi, X) \to (\Psi, Y)$ be a weak morphism. Then,
\begin{enumerate}
	\item If $h(S) \subset S',$ $T_\Psi(h(x)) \leq \tau(x, T_\Phi(x))$ for all $x \in S.$
	\item If $h^{-1} (S') \subset S,$ $\tau(x, T_\Phi(x)) \leq T_\Psi(h(x))$ for all $x \in h^{-1} (S').$
\end{enumerate}
In particular, if $S = h^{-1} (S'),$ we have $T_\Psi(h(x)) = \tau(x, T_\Phi(x))$ for all $x \in S.$
\end{lemma}
\begin{proof}
(1) Let $x \in S.$ As we have $\Phi(x, T_\Phi(x)) \in S,$
\[
	\Psi(h(x), \tau(x, T_\Phi(x))) = h(\Phi(x, T_\Phi(x))) \in h(S) \subset S'.
\]
Because $h(x) \in S',$ it follows that $T_\Psi(h(x)) \leq \tau(x, T_\Phi(x)).$

(2) Let $x \in h^{-1} (S').$  We have $\Psi(h(x), T_\Psi(h(x))) \in S'$ and $T_\Psi(h(x)) = \tau(x, t_x)$ for some $t_x \in \mathbb{R}$ because $\tau(x, -)$ is a homeomorphism. Therefore, we have  
\[
	\Psi(h(x), T_\Psi(h(x))) = h(\Phi(x, t_x)) \in  S',
\]
which implies $\Phi(x, t_x) \in h^{-1} (S') \subset S.$ Thus, we obtain $T_\Phi(x) \leq t_x.$ Because $\tau(x, -)$ is monotonically increasing, we conclude that
\[
	\tau(x, T_\Phi(x)) \leq \tau(x, t_x) = T_\Psi(h(x)),
\]
which is the desired property.
\end{proof}
As a consequence of this lemma, we have the following result.
\begin{lemma}\label{lem_func}
Let $(\Phi, X)$ and $(\Psi, Y)$ be flows with global Poincar\'e sections $S$ and $S',$ respectively, and $(h, \tau): (\Phi, X) \to (\Psi, Y)$ be a weak morphism such that $h^{-1} S' = S.$ Then, we have a morphism of map dynamical systems $h|_S: (P\Phi,S) \to (P\Psi,S'),$ where $h|_S : S \to S'$ is the restriction of $h$ to $S.$ 
\end{lemma}
\begin{proof}
As $P\Phi: S \to S$ and $P\Psi: S' \to S'$ are homeomorphisms, they define map dynamical systems.
For each $x \in S,$ we have 
\[
\begin{aligned}
	h\circ P\Phi(x)&= h(\Phi(x, T_\Phi(x))) \\
				&= \Psi(h(x), \tau(x, T_\Phi(x)))\\
				&= \Psi(h(x), T_\Psi(h(x))) = P \Psi \circ h(x). 
	\end{aligned}
\]
Therefore, $h|_S: (P\Phi,S) \to (P\Psi,S')$ is a morphism of map dynamical systems.
\end{proof}
Thus, we may define the following:
\begin{definition}
Let $(\Phi, X)$ and $(\Psi, Y)$ be flows with global Poincar\'e sections $S$ and $S',$ respectively. A morphism $(h, \tau): (\Phi, X)\to(\Psi, Y)$ in $\bf{WFlow}$ is said to \emph{preserve} the global Poincar\'e sections if $S = h^{-1} (S').$

The \emph{category of flows with global Poincar\'e sections} $\bf{FlowGS}$ is the category whose objects are flows with global Poincar\'e sections and whose morphisms are morphisms in $\bf{Flow}$, which preserves the global Poincar\'e sections. 

Similarly, we may define a category $\bf{WFlowGS}$ whose objects are flows with global Poincar\'e sections and whose morphisms are morphisms in $\bf{WFlow}$, which preserves the global Poincar\'e sections. 

Objects in $\bf{WFlowGS}$ or $\bf{FlowGS}$ are denoted by a triple of the form $(\Phi, X, S),$ where $(\Phi, X)$ is a flow with a global Poincar\'e section $S.$
\end{definition}
From Lemma \ref{lem_func}, we immediately obtain the following: 
\begin{theorem}
The construction of a Poincar\'e map is functorial for $\bf{WFlowGS}.$ That is, there exists a functor $P: \bf{WFlowGS} \to \bf{Map}$ defined by setting
\begin{itemize}
	\item $P(\Phi, X, S) = (P \Phi, S)$ for each object $(\Phi, X, S)$ in $\bf{WFlowGS}.$
	\item For each morphism $h:(\Phi_1, X_1, S_1) \to (\Phi_2, X_2, S_2)$ $P(h) = h|_{S_1}: (P \Phi_1, S_1) \to (P \Phi_2, S_2).$
\end{itemize}
\end{theorem}
\begin{corollary}
The construction of a Poincar\'e map is functorial for $\bf{FlowGS}.$ 
\end{corollary}
\begin{proof}
Take the composition of $P : \bf{WFlowGS} \to \bf{Map}$ with the inclusion functor $\bf{FlowGS}\hookrightarrow \bf{WFlowGS}.$ 
\end{proof}

\section{Poincar\'e maps and suspension flows}\label{sec_ps}
In this section, we consider the categorical relationship between a Poincar\'e map and a suspension. 

To establish the notation, we recall the definition of a suspension flow.
\begin{definition}
Let $f:X \to X$ be a homeomorphism on a topological manifold $X$. The \emph{mapping torus} $X_f$ of $f$ is the manifold defined by
\[
	X_f := X \times \mathbb[0,1]/ \sim, 
\]
where $\sim$ is the smallest equivalence relation with $(x, 1) \sim (f(x), 0)$ for each $x \in X.$ There is a natural surjection $\pi_f: X \times \mathbb[0,1] \to X_f,$ which sends each point to the corresponding equivalence class. 
We denote a point in $X_f$ by $[x,t],$ where  $x \in X,$ $0\leq t <1.$ 
\end{definition}
\begin{definition}
Let $(f, X)$ be a map dynamical system. The \emph{suspension flow} $\Sigma f: X_f \times \mathbb{R} \to X_f$ of $(f, X)$ is defined by
\[
	\Sigma f([x,t], s) := [f^n (x), s+t -n],
\]
where $x \in X,$ $0\leq t <1$ and $n \in \mathbb{Z}$ is a unique integer satisfying $s+t -1 < n \leq s+t.$
\end{definition}
\begin{theorem}
The construction of a suspension flow is functorial. That is, there exists a functor $\Sigma: \bf{Map} \to \bf{FlowGS}$ defined by setting
\begin{itemize}
	\item $\Sigma(f, X) = (\Sigma f, X_f,  (X_f)_0 )$ for each object $(f, X)$ in $\bf{Map}.$
	\item For each morphism $h:(f, X) \to (g, Y),$ we set
	\[\Sigma(h) = \bar h: (\Sigma f, X_f,  (X_f)_0 ) \to (\Sigma g, Y_g,  (Y_g)_0 ),\]
	where $\bar h([x,t]) = [h(x), t]$ and $(X_f)_0 = \{[x,0] \mid x \in X\}$ and $(Y_g)_0 = \{[y,0] \mid y \in Y\}.$
\end{itemize}
\end{theorem}
\begin{proof}
Let $h:(f, X) \to (g, Y)$ be a morphism in $\bf{Map}.$ First, we show that $\bar h: X_f \to Y_g$ is well-defined and continuous. Well-definedness is verified by a direct calculation using $g\circ h = h \circ f.$ The continuity follows from the commutativity of the following diagram and the universal property of the quotient topology:
\[
\begin{diagram}
	\node{X\times [0,1]}
	\arrow{e,t}{\pi_f}
	\arrow{s,t}{h\times \rm{id}}
	\node{X_f}
	\arrow{s,t}{\bar h} \\
	\node{Y\times[0,1]}
	\arrow{e,t}{\pi_g}
	\node{Y_g}
\end{diagram}
\]
We show that $\bar h$ commutes with suspension flows. This is verified by a direct calculation:
\[
\begin{aligned}
	\bar h \left( \Sigma f([x,t], s) \right) &= \bar h \left( [f^n (x), s+t -n] \right)\\
								&= [h (f^n (x)), s+t -n]\\
								&= [g^n(h(x)), s+t -n]\\
								&= \Sigma g([h(x),t], s)\\
								& = \Sigma g(\bar h([x,t]), s)
\end{aligned}
\]
where $x \in X,$ $0\leq t <1$ and $n \in \mathbb{Z}$ is a unique integer satisfying $s+t -1 < n \leq s+t.$

We show that $\bar h$ preserves the sections, that is, $\bar h^{-1}  (Y_g)_0 = (X_f)_0.$ By definition, we have $\bar h (X_f)_0 \subset (Y_g)_0,$ so $(X_f)_0 \subset \bar h^{-1}  (Y_g)_0 .$ Conversely, if $[x,t] \in \bar h^{-1}  (Y_g)_0$ with $0\leq t < 1,$ then $t =0$ and therefore, $[x, t] \in (X_f)_0.$

Finally, we show that $\Sigma$ is a functor. It is clear that $\Sigma(1_{(f,X)}) = 1_{(\Sigma f, X_f,  (X_f)_0 ) }.$ 
If $h_1: (f_1, X_1) \to (f_2, X_2) $ and $h_2: (f_2, X_2) \to (f_3, X_3)$ are morphisms in $\bf{Map},$  we have $\overline{(h_2 \circ h_1)} = \bar h_2 \circ \bar h_1.$
\end{proof}
Now, we have three categories and three functors between them:
\begin{enumerate}
	\item Inclusion functor $I : \bf{FlowGS} \to \bf{WFlowGS}$.
	\item Poincar\'e map functor $P : \bf{WFlowGS} \to \bf{Map}$.
	\item Suspension functor $\Sigma : \bf{Map} \to \bf{FlowGS}$.
\end{enumerate}
From the existence of these functors, we immediately recover some known results on the preservation of isomorphisms.
\begin{theorem}
Each of the following statements holds.
\begin{enumerate}
	\item If two flows with global Poincar\'e sections are topologically equivalent, there is a pair of global Poincar\'e sections such that the Poincar\'e maps are topologically conjugate.
	\item If two maps are topologically conjugate, their suspension flows are topologically conjugate.
\end{enumerate}
\end{theorem}
At this point, we must consider the degree of difference between the original flow and the suspension flow of the Poincar\'e map. First, we note that there is a pair of flows that are topologically equivalent but not topologically conjugate. The following is a modification of an example in \cite{pilyugin2019spaces}.  
\begin{example}
	We define two flows on $A = \{z \in \mathbb{C} \mid 1<|z|< 2\}$ by
		\[
			\begin{aligned}
				\Phi_1(z, t) &:= z e^{i\pi t}\\
				\Phi_2(z, t) &:= z e^{2 \pi it}.
			\end{aligned}
		\]
	Then, they are topologically equivalent but not topologically conjugate.
\end{example}
\begin{proof}
Topological equivalence is obvious.
Suppose there is a homeomorphism $h: A \to A$ such that $\Phi_1$ and $\Phi_2$ are topologically conjugate, that is,
\[
	h\left( z e^{i\pi t}\right)  = h(z) e^{2 \pi i t}
\]	
for all $z \in A$ and $t \in \mathbb{R}.$ By considering $t= 1,$ we obtain $h(z) = h(-z)$ for all $z\in A,$ which contradicts the condition that $h$ is injective.
\end{proof}
Note that we may take $A_0 = \{x \mid 0< x<1\}$ as a global Poincar\'e section for these flows. With this choice, the Poincar\'e map is the identity $\rm{id}_{A_0}$ in either case. 
Further, the suspension flow for $\rm{id}_{A_0}$ coincides with $\Phi_1.$ Thus, the suspension flow of a Poincar\'e map is not necessarily topologically conjugate with the original flow. On the other hand, topological equivalence can be established.

\begin{lemma}\label{lem_adj_1}
There is a natural transformation $(k,\tau) : I \Sigma P I\to I$ defined by the following for each $(\Phi, X, S)$ in $\bf{FlowGS}:$
\[
\begin{aligned}
	k_{(\Phi, X, S)} ([x,t]) &:= \Phi(x, t T_\Phi(x))\\
	\tau_{(\Phi, X, S)} ([x,t],s)&:= \int_0^{s+t } R_{\Phi}(x)(u) d u -tT_\Phi(x),
\end{aligned}
\]
where $x \in S,$ $0\leq t <1$ and  
\[
	R_{\Phi}(x)(u) := \sum_{i \in \mathbb{Z} } T_\Phi( (P\Phi)^i (x)) \chi_{[i, i+1)}(u),
\]
where $\chi_{[i, i+1)}$ is the indicator function of $[i, i+1).$
\end{lemma}
\begin{proof}
First, we show that $(k,\tau)_{ (\Phi, X, S)}: I \Sigma P I(\Phi, X, S) \to (\Phi, X, S)$ is well-defined as a weak morphism in $\bf{WFlowGS}$. Well-definedness and continuity of $k_{(\Phi, X, S)} $ follow from the commutativity of the following diagram:
\[
\begin{diagram}
	\node{S\times [0,1]}
	\arrow{e,t}{\pi_{P\Phi}}
	\arrow{se,r}{K}
	\node{S_{P\Phi}}
	\arrow{s,r}{k_{ (\Phi, X, S)}}\\
	\node{}
	\node{X}
\end{diagram}
\]
where $K:S\times [0,1] \to X$ is defined by $K(x,t) := \Phi(x, t T_\Phi(x))$ for each $(x,t) \in S\times [0,1].$ 

By definition, we have $\tau_{(\Phi, X, S)} ([x,t],0) = 0.$ Because $R_{\Phi}(x)(-) $ is a positive-valued function, $\tau_{(\Phi, X, S)} ([x,t],-):\mathbb{R} \to \mathbb{R}$ is strictly monotonous. By Corollary \ref{div_sum}, it is also a surjection. Thus, $\tau_{(\Phi, X, S)} ([x,t],-)$ is a homeomorphism. 

We check that $k_{(\Phi, X, S)}$ commutes with the flows by a direct calculation. When $n \geq 0,$ where $n \in \mathbb{Z}$ is a unique integer satisfying $s+t -1 < n \leq s+t,$ we have
\[
\begin{aligned}
	k_{(\Phi, X, S)} \left(\Sigma P \Phi([x,t],s) \right)& = k_{(\Phi, X, S)}\left( [(P \Phi)^n (x) ,s+t-n]\right)\\
								& = \Phi((P \Phi)^n (x), (s+t - n) T_\Phi \circ (P \Phi)^n (x))\\
								& = \Phi\left(x, \sum_{i=0}^{n-1}T_\Phi \circ (P \Phi)^i (x) + (s+t - n) T_\Phi \circ (P \Phi)^n (x)\right)\\
								& = \Phi\left(x, \int_0^{s+t } R_{\Phi}(x)(u) d u\right)\\
								& = \Phi\left(k_{(\Phi, X, S)}([x,t]),\tau_{(\Phi, X, S)} ([x,t],s)\right).
\end{aligned}
\]
Noting that $\Phi((P \Phi)^{-1} (x), t) = \Phi(x, t -T_\Phi \circ (P \Phi)^{-1} (x)),$ we calculate the following for $n \leq -1$:
\[
\begin{aligned}
	k_{(\Phi, X, S)} \left(\Sigma P \Phi([x,t],s) \right)& = k_{(\Phi, X, S)}\left( [(P \Phi)^n (x) ,s+t-n]\right)\\
								& = \Phi((P \Phi)^n (x), (s+t - n) T_\Phi \circ (P \Phi)^n (x))\\
								& = \Phi\left(x, \sum_{i=1}^{-n+1}T_\Phi \circ (P \Phi)^{-i} (x) + (s+t - n-1) T_\Phi \circ (P \Phi)^n (x)\right)\\
								& = \Phi\left(x, \int_0^{s+t } R_{\Phi}(x)(u) d u\right)\\
								& = \Phi\left(k_{(\Phi, X, S)}([x,t]),\tau_{(\Phi, X, S)} ([x,t],s)\right).
\end{aligned}
\]
The condition that $k_{(\Phi, X, S)}^{-1} S = (S_{P\Phi})_0$ can be verified by a direct calculation.

Finally, we show that $(k,\tau)$ is natural. Let $h:(\Phi_1, X_1, S_1) \to (\Phi_2, X_2, S_2)$ be a morphism in $\bf{FlowGS}.$
Then, we have
\[
\begin{aligned}
	h\circ k_{(\Phi_1, X_1, S_1)} ([x,t]) &= h \circ \Phi_1(x, t T_{\Phi_1}(x)) \\
								&= \Phi_2(h(x), t T_{\Phi_2}(x)) \\
								&= k_{(\Phi_2, X_2, S_2)} ([h(x),t]) \\
								&= k_{(\Phi_2, X_2, S_2)} \circ \overline{h|_{S_1}}([x,t]), 
\end{aligned}
\]
where $x \in S_1$ and $0 \leq t <1.$
We also have
\[
\begin{aligned}
	\tau_{(\Phi_1, X_1, S_1)} ([x,t], s) &= \int_0^{s+t } R_{\Phi_1}(x)(u) d u -tT_{\Phi_1}(x) \\
								&= \int_0^{s+t } R_{\Phi_2}(h(x))(u) d u -tT_{\Phi_2}(h(x))\\
								& = \tau_{(\Phi_2, X_2, S_2)} (  \overline{h|_{S_1}}( [x,t]), s),
\end{aligned}
\]
using $T_{\Phi_2}(h(x)) = T_{\Phi_1}(x).$
\end{proof}
\begin{remark}\label{rem_bij}
The map $k_{(\Phi, X, S)} $ is bijective. Surjectivity is obvious. For injectivity, if $\Phi(x, t T_\Phi(x)) = \Phi(x', t' T_\Phi(x'))$ for $x, x' \in S$ and $t,t' \in [0,1),$ we have
\[
	x' = \Phi(x, t T_\Phi(x)-t' T_\Phi(x')) \in S.
\]
If $t T_\Phi(x)-t' T_\Phi(x')>0,$ then it follows that $t T_\Phi(x)-t' T_\Phi(x') \geq T_\Phi(x).$ Because $t <1, $ this is a contradiction. Therefore, $t T_\Phi(x)-t' T_\Phi(x') \leq 0.$ By interchanging $t$ and $x$ with $t'$ and $x'$, we also have $t' T_\Phi(x')-t T_\Phi(x) \leq 0.$ Therefore, we conclude that $x = x'$ and consequently $t= t'.$
\end{remark}
Using the invariance of domain theorem, we observe that $(k,\tau) : I \Sigma P I\to I$ is a natural isomorphism. In ordinary terms, this observation can be phrased as follows. 
\begin{corollary}
If $(\Phi, X, S)$ is a flow with a global Poincar\'e section, then $(\Phi, X, S)$ is topologically equivalent to $\Sigma P (\Phi, X, S).$
\end{corollary}
Another natural transformation can be constructed.
\begin{lemma}
There is a natural transformation $l : 1_{\bf{Map}} \to P I \Sigma$ defined by
\[
	l_{(f,X)} (x) : = [x,0]
\]
for each $x \in X.$
\end{lemma}
\begin{proof}
First, we show that $l_{(f,X)} : (f,X) \to P I \Sigma (f,X)$ is well-defined as a morphism in $\bf{Map}.$ As $l_{(f,X)}$ is a composition of continuous maps, it is well-defined and continuous. Additionally, we have
\[
	l_{(f,X)} \circ f(x) = [f(x), 0] = \Sigma f([x,0],1) = P I \Sigma f ([x,0]) =  (P I \Sigma f) \circ l_{(f,X)} (x)
\]
for all $x \in X.$

We show that $l$ is natural. Let $h:(f,X) \to (g, Y)$ be a morphism in $\bf{Map}.$ Then,
\[
	 (P I \Sigma)(h) \circ l_{(f,X)}(x) = [h(x) , 0] =  l_{(g,Y)} \circ h(x)
 \]
 for all $x \in X.$
\end{proof}

These results suggest that there is another category larger than $\bf{FlowGS}$ and smaller than $\bf{WFlowGS}$ for which the constructions of Poincar\'e maps and suspensions become adjoint.
\begin{definition}
A weak morphism $(h,\sigma): (\Phi_1, X_1, S_1) \to (\Phi_2, X_2, S_2)$ in $\bf{WFlowGS}$ is \emph{rate-preserving} if 
\[
\begin{aligned}
	&\sigma\left( \Phi_1(x, t T_{\Phi_1}(x)) , \int_0^{s+t } R_{\Phi_1}(x)(u) d u -tT_{\Phi_1}(x) \right)\\ 
		&= \int_0^{s+t } R_{\Phi_2}(h(x))(u) d u -tT_{\Phi_2}(h(x)) 
	\end{aligned}
\]
for all $x \in S_1,$ $0 \leq t < 1$ and $s \in \mathbb{R},$ where $R_{\Phi_1}$ and  $R_{\Phi_2}$ are the same as  in Lemma \ref{lem_adj_1}.
\end{definition}
\begin{lemma}
If $(h,\sigma): (\Phi_1, X_1, S_1) \to (\Phi_2, X_2, S_2)$ in $\bf{WFlowGS}$ is rate-preserving, we have
\[
	\sigma(x, t T_\Phi(x)) = t \sigma(x, T_\Phi(x))
\]
for all $x \in S_1$ and $0 \leq t < 1.$
\end{lemma}
\begin{proof}
We show this by a direct calculation. Let $x \in S_1$ and $0 \leq t < 1.$ Then, we have
\[
	\begin{aligned}
		\sigma(x, t T_{\Phi_1}(x)) &= \sigma\left( \Phi_1(x, 0\cdot T_{\Phi_1}(x)) , \int_0^{t+0 } R_{\Phi_1}(x)(u) d u -0\cdot T_{\Phi_1}(x) \right) \\
							&= \int_0^{t+0 } R_{\Phi_2}(h(x))(u) d u -0\cdot T_{\Phi_2}(h(x)) \\
							& = t T_{\Phi_2}(h(x)) = t \sigma(x, T_{\Phi_1}(x)).
	\end{aligned}
\]
Here, we used the result of Lemma \ref{lem_period}.
\end{proof}
\begin{lemma}
The identity morphism in $\bf{WFlowGS}$ is rate-preserving. The composition of two rate-preserving morphisms is again rate-preserving. 
\end{lemma}
\begin{proof}
The first statement is obvious. Let $(h_1,\sigma_1): (\Phi_1, X_1, S_1) \to (\Phi_2, X_2, S_2)$ and $(h_2,\sigma_2): (\Phi_2, X_2, S_2) \to (\Phi_3, X_3, S_3)$ be  rate-preserving morphisms. Then, we have
\[
\begin{aligned}
	&\sigma_2 \circ \sigma_1  \left(\Phi_1(x, t T_{\Phi_1}(x)) , \int_0^{s+t } R_{\Phi_1}(x)(u) d u -tT_{\Phi_1}(x) \right)\\
	& = \sigma_2 \left(h_1\left(\Phi_1(x, t T_{\Phi_1}(x))  \right), \sigma_1\left( \Phi_1(x, t T_{\Phi_1}(x)) , \int_0^{s+t } R_{\Phi_1}(x)(u) d u -tT_{\Phi_1}(x) \right) \right)\\
	& = \sigma_2 \left(h_1\left(\Phi_1(x, t T_{\Phi_1}(x))\right),  \int_0^{s+t } R_{\Phi_2}(h_1(x))(u) d u -tT_{\Phi_2}(h_1(x))  \right)\\
	& = \sigma_2 \left(\Phi_2(h_1(x), \sigma_1(x, t T_{\Phi_1}(x))),  \int_0^{s+t } R_{\Phi_2}(h_1(x))(u) d u -tT_{\Phi_2}(h_1(x))  \right)\\
	& = \sigma_2 \left(\Phi_2(h_1(x), t\sigma_1(x, T_{\Phi_1}(x))),  \int_0^{s+t } R_{\Phi_2}(h_1(x))(u) d u -tT_{\Phi_2}(h_1(x))  \right)\\
	& = \int_0^{s+t } R_{\Phi_3}(h_2\circ h_1(x))(u) d u -tT_{\Phi_3}(h_2\circ h_1(x)),
	\end{aligned}
\]
for all $x \in S_1,$ $0 \leq t < 1$ and $s \in \mathbb{R}.$
\end{proof}
Therefore, we can define a category $\bf{RWFlowGS},$ whose objects are flows with global Poincar\'e sections and whose morphisms are rate-preserving morphisms. If we denote the inclusion functors
by $J^-: \bf{FlowGS} \to \bf{RWFlowGS}$ and $J^+: \bf{RWFlowGS}  \to \bf{WFlowGS}$, it is clear that $I = J^+J^-.$

\begin{lemma}
Let $(\Phi, X, S)$ be an object in $\bf{FlowGS}.$ Then, the weak morphism $(k,\tau)_{ (\Phi, X, S)}: I \Sigma P I(\Phi, X, S) \to (\Phi, X, S)$ is rate-preserving.
\end{lemma}
\begin{proof}
First, we note that 
\[
	 \int_0^{s+t } R_{\Sigma P \Phi}([x, 0])(u) d u -tT_{\Sigma P \Phi}([x,0]) = s 
\]
for all $[x,0] \in (S_{P \Phi})_0,$ $0 \leq t <1$ and $s \in \mathbb{R}$ because $T_{\Sigma P \Phi}([x,0]) = 1$ for all  $[x,0] \in (S_{P \Phi})_0.$
We calculate the following:
\[
\begin{aligned}
	&\tau_{(\Phi, X, S)}\left( \Sigma P \Phi([x,0], t T_{\Sigma P \Phi}([x,0])) , \int_0^{s+t } R_{\Sigma P \Phi}([x,0])(u) d u -tT_{\Sigma P\Phi}([x,0]) \right) \\
	&= \tau_{(\Phi, X, S)}\left( \Sigma P \Phi([x,0], t ) , s \right) \\
	& = \tau_{(\Phi, X, S)}\left( [x,t], s \right)\\
	& = \int_0^{s+t } R_{ \Phi}(x)(u) d u -tT_{\Phi}(x)\\
	&= \int_0^{s+t } R_{ \Phi}(k_{(\Phi, X, S)}([x,0]))(u) d u -tT_{\Phi}(k_{(\Phi, X, S)}([x,0])),
\end{aligned}
\]
for all  $[x,0] \in (S_{P \Phi})_0,$  $0 \leq t < 1$ and $s \in \mathbb{R}.$
\end{proof}
Thus, we have the following result:
\begin{lemma}
There is a natural transformation $(k,\tau) : J^- \Sigma P J^+\to 1_{\bf{RWFlowGS}}$ given 
by the restriction of the natural transformation $(k,\tau) : I \Sigma P I\to I$ defined in Lemma \ref{lem_adj_1}.
\end{lemma}
\begin{proof}
It is sufficient to verify the naturality conditions. Let $(h,\sigma):(\Phi_1, X_1, S_1) \to (\Phi_2, X_2, S_2)$ be a morphism in $\bf{RWFlowGS}.$

If $[x,t] \in (S_1)_{P\Phi_1}$ with $0 \leq t <1,$ we have
\[
	\begin{aligned}
		h \circ k_{(\Phi_1, X_1, S_1)} ([x,t]) &= h\left( \Phi_1(x, t T_{\Phi_1}(x))\right)\\
										&= \Phi_2\left( h(x), \sigma(x, t T_{\Phi_1})(x)\right)\\
										&= \Phi_2\left( h(x), t\sigma(x, T_{\Phi_1})(x)\right)\\
										&= \Phi_2\left( h(x), tT_{\Phi_2}(x)\right)\\
										& = k_{(\Phi_2, X_2, S_2)} ([h(x),t]) \\
										& =  k_{(\Phi_2, X_2, S_2)}\circ \bar h ([x,t]).
	\end{aligned}
\]
Further, for all $[x,t] \in (S_1)_{P\Phi_1}$ with $0 \leq t <1$ and  $s \in \mathbb{R},$
\[
	\begin{aligned}
		&\sigma \circ \tau_{(\Phi_1, X_1, S_1)} ([x,t],s) \\
		&=\sigma\left(k_{(\Phi_1, X_1, S_1)} ([x,t]), \tau_{(\Phi_1, X_1, S_1)} ([x,t],s)  \right)\\
		&=\sigma\left(\Phi_1(x, t T_{\Phi_1}(x)), \int_0^{s+t } R_{\Phi_1}(x)(u) d u -tT_{\Phi_1}(x) \right)\\
		&=\int_0^{s+t } R_{\Phi_2}(h(x))(u) d u -tT_{\Phi_2}(h(x))\\
		& = \tau_{(\Phi_2, X_2, S_2)} ([h(x),t],s)\\
		&= \tau_{(\Phi_2, X_2, S_2)} \circ {\rm id} ([x,t],s),
	\end{aligned}
\]
where ${\rm id}$ is the time part of $J^-(\bar h).$
\end{proof}
Combining the results above, we obtain the desired result, which gives us the exact relation between the constructions of Poincar\'e maps and suspension flows.
\begin{theorem}
$J^- \Sigma \dashv P J^+.$
\end{theorem}
\begin{proof}
We verify that the triangle identities are satisfied by $l : 1_{\bf{Map}} \to P I \Sigma= (P J^+)(J^- \Sigma)$ and $(k,\tau) : (J^- \Sigma)( P J^+)\to 1_{\bf{RWFlowGS}}.$

In what follows, we omit $J^+$ or $J^-$ for ease of notation. 

Let $(f,X)$ be an object in $\bf{Map}.$ Then, we have
\[
\begin{aligned}
	&k_{\Sigma(f,X)} \circ \Sigma(l_{(f,X)}) ([x,t]) \\
	&=k_{\Sigma(f,X)} \left( [l_{(f,X)}(x) , t]\right)\\
	& = \Sigma f \left( l_{(f,X)}(x), t\right)\\
	& = [x,t]
\end{aligned}
\]
for all $x\in X$ and $0 \leq t < 1.$ Further, we have
\[
\begin{aligned}
	&\tau_{\Sigma(f,X)} \circ {\rm id} ([x,t],s)\\
	& = \tau_{\Sigma(f,X)} \left(  [l_{(f,X)}(x) , t], s\right)\\
	& = \int_0^{s+t } R_{\Sigma f}(l_{(f,X)}(x))(u) d u -tT_{\Sigma f}(l_{(f,X)}(x))\\
	& = s
\end{aligned}
\]
for all $x\in X,$ $0 \leq t < 1$ and $s \in \mathbb{R}.$
These results  show that the following diagram commutes in $\bf{RWFlowGS}.$
\[
\begin{diagram}
	\node{\Sigma(f,X)}
	\arrow{e,t}{\Sigma\left( l_{(f,X)}\right)}
	\arrow{se,r}{1_{\Sigma(f,X)}}
	\node{\Sigma P\Sigma(f,X)}
	\arrow{s,r}{(k,\tau)_{\Sigma(f,X)}}\\
	\node{}
	\node{\Sigma(f,X)}
\end{diagram}
\]

Let $(\Phi, X, S)$ be an object in $\bf{RWFlowGS}.$ Then, we have 
\[
\begin{aligned}
	& P\left((k,\tau)_{(\Phi, X, S)}\right) \circ l_{P(\Phi, X, S)} (x)\\
	&= k_{(\Phi, X, S)} |_{(S_{P\Phi})_0} \left( [x, 0]\right)\\
	& = k_{(\Phi, X, S)}\left( [x, 0]\right)\\
	& = \Phi(x, 0) = x
\end{aligned}
\]
for all $x \in S.$ Therefore, the following diagram commutes in $\bf{Map}.$
\[
\begin{diagram}
	\node{P(\Phi, X, S)}
	\arrow{e,t}{ l_{P(\Phi, X, S)}}
	\arrow{se,r}{1_{P(\Phi, X, S)}}
	\node{P\Sigma P(\Phi, X, S)}
	\arrow{s,r}{P\left((k,\tau)_{(\Phi, X, S)}\right)}\\
	\node{}
	\node{P(\Phi, X, S)}
\end{diagram}
\]
Thus, we conclude that $J^- \Sigma \dashv P J^+.$
\end{proof}
The next corollary is an immediate consequence of Remark \ref{rem_bij} and the injectivity of $l_{(f,X)}.$ 
\begin{corollary}\label{cor_eq}
The categories $\bf{Map}$ and $\bf{RWFlowGS}$ are equivalent.
\end{corollary}
We remark that the rate-preserving condition can always be assumed for topologically equivalent flows.
\begin{theorem}\label{thm_eq}
Let $(\Phi_1, X_1, S_1)$ and $(\Phi_2, X_2, S_2)$ be isomorphic in $\bf{WFlowGS}.$ Then, they are isomorphic in $\bf{RWFlowGS}.$
\end{theorem}
\begin{proof}
By the functoriality of $\Sigma P,$ $\Sigma P(\Phi_1, X_1, S_1)$ and $\Sigma P(\Phi_2, X_2, S_2)$ are isomorphic in $\bf{RWFlowGS}.$ Because $(k,\tau)$ gives isomorphisms, we conclude that $(\Phi_1, X_1, S_1)$ and $(\Phi_2, X_2, S_2)$ are isomorphic in $\bf{RWFlowGS}.$
\end{proof}
\section{Concluding  remarks}\label{sec_conclude}
The categorical equivalence of Corollary \ref{cor_eq} enables us to obtain correspondences between various concepts of flows and map dynamical systems. For example, Theorem \ref{thm_eq} implies that the topological conjugacy of map dynamical systems is categorically equivalent to the topological equivalence of flows. This provides further justification for the use of topological equivalence in the study of flows, in addition to the usual argument that topological conjugacy is too strict. 

We also observe a lack of correspondence for some notions. As flows with global Poincar\'e sections do not have equilibria, it follows that map dynamical systems do not have a concept corresponding to them under the equivalence obtained here. It would be interesting to consider whether there exists another pair of functors under which fixed points correspond to equilibria. A candidate will be the time-one map because it corresponds to the discretization functor, which has been considered in \cite{cestau2017prolongations}. However, it is known that this construction is not very expressive, and it is unclear whether an interesting equivalence can be found \cite{bonomo2020continuous}. 

\section*{Acknowledgements}
This study was supported by a Grant-in-Aid for JSPS Fellows (20J01101). This work was also supported by the Research Institute for Mathematical Sciences, an International Joint Usage/Research Center located in Kyoto University.
\bibliographystyle{plain}
\bibliography{poincare_suspension}
\end{document}